\documentclass{amsart}

\usepackage{amssymb,amscd,pb-diagram}
\usepackage[all,arc]{xy}
\usepackage[english]{babel}
\usepackage{tabularx}

\usepackage{tabularx}
\usepackage{rotating}
\usepackage[all]{xy}
\usepackage{hypbmsec}
\usepackage{graphicx}
\usepackage{tikz}
\usetikzlibrary{angles}
\usetikzlibrary{patterns}
\usetikzlibrary{arrows.meta}
\usetikzlibrary{plotmarks}
\usepackage{float}

\newtheorem{theorem}{Theorem}[section]
\newtheorem{lemma}{Lemma}[section]
\newtheorem{corollary}{Corollary}[section]
\newtheorem{proposition}{Proposition}[section]
\newtheorem{definition}{Definition}
\newtheorem{remark}{Remark}[section]
\numberwithin{equation}{section}

\def\RR{\mathbb R}
\def\CC{\mathbb C}
\def\XX{\mathfrak X}

\title{The Hartogs phenomenon in almost homogeneous varieties}

\author{Sergei~V.~Feklistov}
\date{}

\begin{document}
\maketitle
\begin{abstract}
We study the Hartogs extension phenomenon in noncompact almost homogeneous algebraic varieties and we prove the cohomological and weight criteria for the Hartogs phenomenon. In the case of spherical varieties, we prove a criterion for the Hartogs phenomenon in terms of colored fans.

Bibliography: 30~titles.
\end{abstract}

\footnotetext[0]{This work is supported by the Krasnoyarsk Mathematical Center and financed by the Ministry of
Science and Higher Education of the Russian Federation (Agreement No. 075-02-2022-876).}

\section{Introduction}
\label{s1}

The classical Hartogs extension theorem states that for every domain $W\subset\CC^{n}(n>1)$ and a compact set $K\subset W$ such that $W\setminus K$ is connected, the restriction homomorphism $$H^{0}(W,\mathcal{O})\to H^{0}(W\setminus K, \mathcal{O})$$ is an isomorphism.

A natural question arises if this is true for complex analytic spaces. 

\begin{definition}
We say that a noncompact connected complex analytic space $X$ admits the Hartogs phenomenon if for any domain $W\subset X$ and a compact set $K\subset W$ such that $W\setminus K$ is connected, the restriction homomorphism $$H^{0}(W,\mathcal{O})\to H^{0}(W\setminus K, \mathcal{O})$$ is an isomorphism. 
\end{definition}

In this or a similar formulation this phenomenon has been extensively studied in many situations, including Stein manifolds and spaces, $(n-1)$-complete normal complex spaces and so on \cite{Andersson,AndrGrau,AndrHill,AndrVesen,BanStan,ColtRupp,Harvey,Merker,Rossi,Vassiliadou,Viorel}.

Our goal is to study the Hartogs phenomenon in almost homogeneous algebraic varieties. Let $X$ be a complex analytic variety (i.e. reduced, irreducible complex analytic space) and $G$ be a connected complex Lie group acting holomorphically on $X$. In this situation $X$ is called complex analytic \textit{$G$-variety}. A complex analytic $G$-variety $X$ is called almost homogeneous if $X$ has an open $G$-orbit $\Omega$ \cite{Akh}. 

For example, toric varieties, horospherical varieties, flag varieties (moreover, each spherical variety) are almost homogeneous algebraic varieties. In the context of toric varieties the Hartogs phenomenon was studied in \cite{Fekl, Marci, Marci3}. 

We shall follow an approach that goes back to Serre \cite{Serre}. First we prove the result about cohomology vanishing for some class of complex analytic varieties. 

 	\begin{theorem}
 	\label{th1I}
 	Let $X$ be a noncompact complex analytic variety possessing the following properties: 
	\begin{enumerate}
	\item $X$ admits an open embbeding (not necessary with dense image) $X\hookrightarrow X'$ into a complex analytic variety $X'$ with $H^{1}(X',\mathcal{O})=0$;
	\item $X$ admits a compact exhaustion $\{V_n\}$ with $X\setminus V_n$ being connected sets. 
	\end{enumerate}
 	Then $X$ admits the Hartogs phenomenon if and only if $H^{1}_{c}(X,\mathcal{O})=0$.
	\end{theorem}
	
We apply Theorem \ref{th1I} to \textit{normal} noncompact complex analytic varieties that possess some properties related to their compactifications. First we give the following definition.

\begin{definition}\label{maindef}
A noncompact complex analytic variety $X$ is called \textit{$(b,\sigma)$-com\-pa\-cti\-fiable} if it admits a compactification $X'$ with the following properties:
\begin{enumerate}
\item $X'$ is a compact complex analytic variety;
\item $X'\setminus X$ is a proper analytic subset and it has $b$ connected components;
\item $\dim_{\CC} H^{1}(X',\mathcal{O})=\sigma$.
\end{enumerate} 
If $X$ is normal (resp. algebraic), then also we require $X'$ to be normal (resp. algebraic). If $X$ is a $G$-variety, then we requre $X'$ to be a $G$-variety and the compactification map to be $G$-equivariant. 
\end{definition}

Note that the number $\sigma$ is called irregularity of $X'$ (see \cite{Hart} in the context of projective surfaces). The number $b$ is related to the number $e(X)$ of topological ends of $X$ (see \cite{Pesch} about topological ends and \cite{Viorel} about the relation with the Hartogs phenomenon). Namely, we have $b\leq e(X)$.
 
In this paper we consider only the case $b=1$ and $\sigma=0$. This means that $X'\setminus X$ is connected and $H^{1}(X',\mathcal{O})=0$. 

For a (1,0)-compactifiable complex analytic variety $X$ we have the canonical isomorphism $$H^{1}_{c}(X,\mathcal{O})=H^{0}(Z,i^{-1}\mathcal{O})/\CC,$$ where $Z=X'\setminus X$ and $i\colon Z\hookrightarrow X'$ is the closed embedding. It follows that we have to study the space $H^{0}(Z,i^{-1}\mathcal{O})$.
	
Now let $X$ be a \textit{normal} (1,0)-compactifiable \textit{almost homogeneous} \textit{algebraic} $G$-variety, where $G$ is acting algebraically on $X$, let $X'$ be a corresponding compactification of $X$. 

Denote by $\mathcal{G}(X')$ the set of $G$-stable prime divisors on $X'$ and define $$Y:=X'\setminus \bigcup\limits_{D\in\mathcal{G}(X'),D\subset X} D.$$ 

Note that the set $\{D\in\mathcal{G}(X')\mid D\subset X\}$ may be the empty set. 

Let $G$ be a connected complex \textit{reductive} Lie group, $B\subset G$ be a Borel subgroup, $T\subset B$ be a maximal algebraic torus with the character lattice $\XX(T)$, the 1-parameter subgroup lattice $\XX^{*}(T):=Hom(\XX(T),\mathbb{Z})$ and the set of dominant characters $\XX_{+}(T)$. 

In Proposition \ref{Prop5} we prove that the algebra of regular functions $\CC[Y]$ is a dense subspace of the topological vector space $H^{0}(Z,i^{-1}\mathcal{O})$ (with the direct limit topology). It follows that $H^{1}_{c}(X,\mathcal{O})=0$ if and only if $\CC[Y]=\mathbb{C}$. 

Since $G$ acts algebraically on $Y$, it follows that $\CC[Y]$ is a representation of $G$. Now define the weight monoid of $Y$ as $$\Lambda_{+}(Y):=\{\lambda\in \XX_{+}(T)\mid \CC[Y]_{\lambda}^{(B)}\neq 0\}$$
where $\CC[Y]_{\lambda}^{(B)}:=\{f\in\CC[Y]\mid \exists\lambda\in\XX(T): b.f=\lambda(b)f,\forall b\in B\}$. 

We obtain the following weight criterion.
\begin{theorem}
Let $G$ be a connected complex reductive Lie group, and let $X$ be a normal (1,0)-compactifiable almost homogeneous algebraic $G$-variety. Then $X$ admits the Hartogs phenomenon if and only if $\Lambda_{+}(Y)=0$ and $\CC[Y]^{B}=\CC$.
\end{theorem}

In this paper we consider so-called spherical $G$-varieties. They are almost homogeneous algebraic $G$-varieties where $G$ is a complex reductive group and a Borel subgroup $B\subset G$ acts on $X$ with an open orbit. 

Let $X$ be a spherical $G$-variety with an open $G$-orbit $\Omega$. Let us define the weight lattice as $$M:=\{\lambda\in \XX(T)\mid \CC(\Omega)^{(B)}_{\lambda}\neq0\}.$$ 
Note that $M$ is a sublattice in the character lattice $\XX(T)$. Let $M_{\RR}:=M\otimes\RR$. Denote by $\mathcal{B}(Y)$ the set of all $B$-stable prime divisors on $Y$. 

Each $B$-stable divisor $D\in \mathcal{B}(Y)$ defines the discrete valuation $$v_{D}\colon \CC(\Omega)\setminus\{0\}\to \mathbb{Z}.$$ Recall that $v_{D}(f)$ is the order of zeros or poles of $f$ at $D$. Also the valuation $v_{D}$ defines a point in the dual weight lattice $a_{D}\in N:=Hom(M,\mathbb{Z})$ by setting $\langle a_{D},\lambda\rangle:=v_{D}(f)$ for $f\in \CC(\Omega)^{(B)}_{\lambda}$.

Define the following cone in the space $N_{\RR}$: 
$$C:=\mathbb{R}_{\geq 0}\big\langle a_{D}\mid D\in \mathcal{B}(Y)\big\rangle.$$

We have the following convex geometric criterion of the Hartogs phenomenon.

\begin{corollary}
Let $X$ be a (1,0)-compactifiable spherical variety. Then $X$ admits the Hartogs phenomenon if and only if $C=N_{\RR}$.
\end{corollary}

Note that this criterion can also be formulated in terms of colored fans. Briefly, each spherical variety with the open $G$-orbit $\Omega$ is encoded by a colored fan --- a collection of strictly convex cones in a real vector space with the common apex that may intersect only along their common faces. For more details see Section \ref{s6} in this paper or \cite{Gandini, Tim}. 

Let $X_{\Sigma}$ be a spherical variety with a colored fan $\Sigma$. For $\Sigma$ we denote by $|\Sigma|$ the support of $\Sigma$ and denote by $\overline{N_{\RR}\setminus |\Sigma|}$ the closure of $N_{\RR}\setminus |\Sigma|$ in $N_{\RR}$. 

By $\mathcal{V}(\Omega)$ we denote the finitely generated convex rational cone in $N_{\mathbb{Q}}=N\otimes\mathbb{Q}$ of all $G$-invariant valuations (see \cite[Sections 4, 10]{Gandini} or Section \ref{s6} in this paper), and by $\mathcal{V}_{\mathbb{R}}(\Omega)$ we denote the cone generated by the set $\mathcal{V}(\Omega)$ in $N_{\mathbb{R}}$. By $\mathcal{B}(\Omega)$ we denote the set of all prime $B$-stable divisors of $\Omega$. 

Let us note that a noncompact spherical variety $X_{\Sigma}$ is  (1,0)-compactifiable if and only if $\mathcal{V}_{\mathbb{R}}(\Omega)\setminus |\Sigma|$ is a connected set (see Lemma \ref{lemconnect}). We have the following main result.

\begin{theorem}
Let $X_{\Sigma}$ be a noncompact spherical $G$-variety with the open $G$-orbit $\Omega$ and with the colored fan $\Sigma$ such that $\mathcal{V}_{\mathbb{R}}(\Omega)\setminus |\Sigma|$ is a connected set. Then $X_{\Sigma}$ admits the Hartogs phenomenon if and only if $$\mathbb{R}_{\geq 0}\langle\overline{\mathcal{V}_{\mathbb{R}}(\Omega)\setminus |\Sigma|}\cup \{a_{D}\mid D\in \mathcal{B}(\Omega)\}\rangle=N_\mathbb{R}.$$
\end{theorem}  

In particular, let $X_{\Sigma}$ be a noncompact horospherical variety with the open $G$-orbit $\Omega$. Let $U$ is the unipotent radical of the Borel subgroup $B$, $S$ be the set of simple roots with respect to $B$ and $S^{\vee}$ be the set of dual simple roots. Let $H$ be the stabilizer of any point $o\in\Omega$ such that $H\supset U^{-}$. Consider a parabolic subgroup $P\supset B$ such that $P^{-}=N_{G}(H)$. Note that parabolic subgroups containing a given Borel subgroup $B$ are parametrized by subsets of simple roots $I\subset S$. Let $I$ be a subset of $S$ which corresponds to the parabolic subgroup $P$.

The injective map $\iota\colon M_\RR\hookrightarrow \XX(T)\otimes\RR$ induces the surjective map $$\iota^{*}\colon \XX^{*}(T)\otimes\RR\twoheadrightarrow N_{\RR}:=N\otimes\RR.$$ 

We obtain the following corollary.

\begin{corollary}
Let $X_{\Sigma}$ be a noncompact horospherical $G$-variety with the open $G$-orbit $\Omega$ and with the colored fan $\Sigma$ such that $N_{\mathbb{R}}\setminus |\Sigma|$ is a connected set. Then $X_{\Sigma}$ admits the Hartogs phenomenon if and only if $$\mathbb{R}_{\geq 0}\langle\overline{N_{\mathbb{R}}\setminus |\Sigma|}\cup \iota^{*}((S\setminus I)^{\vee})\rangle=N_\mathbb{R}$$ where $(S\setminus I)^{\vee}=\{\alpha^{\vee}\mid\alpha\in S\setminus I\}$.
\end{corollary} 

\section{Cohomological criterion for the Hartogs phenomenon}
\label{s2}

We consider only reduced, irreducible complex analytic spaces (i.e. complex analytic \textit{varieties}). 

J.-P. Serre proved the Hartogs phenomenon for Stein manifolds by using triviality of the cohomology group with compact supports $H^{1}_{c}(X,\mathcal{O})$ where $\mathcal{O}$ is the sheaf of holomorphic functions \cite{Serre}. For more details about of the cohomology group with compact supports of sheaves see \cite{BanStan}.

We consider a class of complex analytic varieties such that for this class the triviality of the group $H^{1}_{c}(X,\mathcal{O})$ is a necessary and sufficient condition for the Hartogs phenomenon. 

 \begin{theorem}
 	\label{th1}
 	Let $X$ be a noncompact complex analytic variety with the following properties: 
	\begin{enumerate}
	\item $X$ admits an open embbeding (not necessary with dense image) $X\hookrightarrow X'$ into a complex analytic variety $X'$ with $H^{1}(X',\mathcal{O})=0$;
	\item $X$ admits a compact exhaustion $\{V_n\}$ with $X\setminus V_n$ being connected sets.  
	\end{enumerate}
 	Then $X$ admits the Hartogs phenomenon if and only if $H^{1}_{c}(X,\mathcal{O})=0$.
\end{theorem}

\begin{proof}
First, we assume $H^{1}_{c}(X,\mathcal{O})=0$. Let $W\subset X$ be a domain and $K\subset W$ be a compact set such that $W\setminus K$ is connected. Note that since $X, W, W\setminus K$ are connected sets, it follows that $X\setminus K$ is connected set.

We need the following Lemma. 
\begin{lemma}\label{Lemma1}
Let $K\subset X$ be a compact set such that $X\setminus K$ is connected. If $H^{1}_{c}(X,\mathcal{O})=0$, then the restriction homomorphism $H^{0}(X,\mathcal{O}) \to H^{0}(X\setminus K,\mathcal{O})$ is an isomorphism.
\end{lemma}

\begin{proof}

We consider the following exact sequence of cohomology groups \cite{BanStan}:

\begin{equation*}
	\xymatrix@C=0.5cm{
0 \ar[r] &  H^{0}_{K}(X,\mathcal{O}) \ar[r] &  H^{0}(X,\mathcal{O}) \ar[r]^{R_{K}} & H^{0}(X\setminus K,\mathcal{O}) \ar[r]^{F_{K}} & H^{1}_{K}(X,\mathcal{O}) \ar[r] & \cdots }
\end{equation*}

Obviously, $H^{0}_{K}(X,\mathcal{O})=0$. So, the restriction homomorphism $R_{K}$ is injective. 

Recall that if $S,T\subset X$ are compact sets and $S\subset T$, then we obtain the canonical homomorphism $\phi_{ST}\colon H^{1}_{S}(X,\mathcal{O})\to H^{1}_{T}(X,\mathcal{O})$. Moreover, we have the canonical isomorphism \cite{BanStan} $$\varinjlim\limits_{S} H^{1}_{S}(X,\mathcal{O})\cong H^{1}_{c}(X,\mathcal{O})$$ where the inductive limit is taken over all compact subsets $S$ of $X$ (or over a cofinal part of them).

Let $f\in H^{0}(X\setminus K,\mathcal{O})$ and we consider $ F_{K}(f)\in H^{1}_{K}(X,\mathcal{O})$. Since $H^{1}_{c}(X,\mathcal{O})=0$, it follows that there exists a compact set $K'\subset X$ such that $K\subset K'$ and $\phi_{KK'}(F_{K}(f))=0$.

Note that the following diagram is commutative.

\[ \begin{diagram} 
\node{0} \arrow{e,t}{} \node{H^{0}(X,\mathcal{O})} \arrow[2]{e,t}{R_{K}} \arrow{ese,r}{R_{K'}} 
\node[2]{H^{0}(X\setminus K,\mathcal{O})} \arrow[2]{e,t}{F_{K}} \arrow{s,r}{}
\node[2]{H^{1}_{K}(X,\mathcal{O})}  \arrow{s,r}{\phi_{KK'}} \\
\node[4]{H^{0}(X\setminus K',\mathcal{O})} \arrow[2]{e,t}{F_{K'}} 
\node[2]{H^{1}_{K'}(X,\mathcal{O})} 
\end{diagram}\]

It follows that $F_{K'}(f|_{X\setminus K'})=0$. Thus there exists $g\in H^{0}(X,\mathcal{O})$ such that $g|_{X\setminus K'}=f|_{X\setminus K'}$. Recall that $X\setminus K$ is a connected set. Using the uniqueness theorem we obtain $g|_{X\setminus K}=R_{K}(g)=f$.  The proof of the lemma is complete.
\end{proof}

Further on, we have the following commutative diagram for $K\subset W\subset X\subset X'$. 

\[
\begin{diagram} 
\node{0\longrightarrow H^{0}(X',\mathcal{O})} \arrow{e,t}{R_1} \arrow{s,r}{}
\node{H^{0}(X'\setminus K,\mathcal{O})} \arrow{e,t}{} \arrow{s,r}{}
\node{H^{1}_{K}(X',\mathcal{O})} \arrow{e,t}{} \arrow{s,r}{H_1} 
\node{H^{1}(X',\mathcal{O})}\arrow{s,r}{}
\\
\node{0\longrightarrow H^{0}(X,\mathcal{O})} \arrow{e,t}{R_{2}} \arrow{s,r}{}
\node{H^{0}(X\setminus K,\mathcal{O})} \arrow{e,t}{} \arrow{s,r}{}
\node{H^{1}_{K}(X,\mathcal{O})} \arrow{e,t}{} \arrow{s,r}{H_2} 
\node{H^{1}(X,\mathcal{O})}\arrow{s,r}{}
\\
\node{0\longrightarrow H^{0}(W,\mathcal{O})} \arrow{e,t}{R_3} 
\node{H^{0}(W\setminus K,\mathcal{O})} \arrow{e,t}{} 
\node{H^{1}_{K}(W,\mathcal{O})} \arrow{e,t}{}
\node{H^{1}(W,\mathcal{O})}
\end{diagram}
\]

Using Lemma \ref{Lemma1}, we obtain that the restriction homomorphism $R_2$ is an isomorphism. Since $X\cap (X'\setminus K)=X\setminus K$, it follows that $R_{1}$ is an isomorphism. 

Since $H^{1}(X',\mathcal{O})=0$, it follows that $H^{1}_{K}(X',\mathcal{O})=0$. Using the excision property \cite{BanStan} we obtain $H_{1}$ and $H_{2}$ are canonical isomorphisms. We obtain that $H^{1}_{K}(W,\mathcal{O})=0$ and the restriction homomorphism $R_{3}$ is an isomorphsim.

Now, we assume $X$ admits the Hartogs phenomenon. In particular, the restricition homomorphism $H^{0}(X,\mathcal{O})\to H^{0}(X\setminus V_{n},\mathcal{O})$ is an isomorphism, where $\{V_{n}\}$ is the compact exhaustion such that $X\setminus V_n$ is a connected set for any $n$. 

We have the following commutative diagramm for $V_{n}\subset X\subset X'$. 

\[
\begin{diagram} 
\node{0\longrightarrow H^{0}(X',\mathcal{O})} \arrow{e,t}{R_1} \arrow{s,r}{}
\node{H^{0}(X'\setminus V_{n},\mathcal{O})} \arrow{e,t}{} \arrow{s,r}{}
\node{H^{1}_{V_{n}}(X',\mathcal{O})} \arrow{e,t}{} \arrow{s,r}{H_1} 
\node{H^{1}(X',\mathcal{O})}\arrow{s,r}{}
\\
\node{0\longrightarrow H^{0}(X,\mathcal{O})} \arrow{e,t}{R_{2}} 
\node{H^{0}(X\setminus V_{n},\mathcal{O})} \arrow{e,t}{} 
\node{H^{1}_{V_{n}}(X,\mathcal{O})} \arrow{e,t}{} 
\node{H^{1}(X,\mathcal{O})}
\end{diagram}
\]

Since $R_{2}$ is an isomorphism so is $R_1$ too. Since $H^{1}(X',\mathcal{O})=0$, it follows that $H^{1}_{V_{n}}(X,\mathcal{O})\cong H^{1}_{V_{n}}(X',\mathcal{O})=0$. By assumption, $\{V_{n}\}$ is a compact exhaustion of $X$. Thus we obtain $H^{1}_{c}(X,\mathcal{O})\cong \varinjlim\limits_{n} H^{1}_{V_n}(X,\mathcal{O})=0$.  The proof of the theorem is complete.
\end{proof}

Let $X$ be a normal (1,0)-compactifiable complex analytic variety, $X'$ be a corresponding compactification of $X$ and let $Z:=X'\setminus X$. Thus we have the following corollary.

\begin{corollary}\label{Cor1}
A normal (1,0)-compactifiable complex analytic variety $X$ admits the Hartogs phenomenon if and only if 
$$H^{0}(Z,i^{-1}\mathcal{O})=\CC$$ where $i\colon Z\hookrightarrow X'$ is the closed embedding.
\end{corollary}

\begin{proof}
First, we have the following exact sequence \cite[Chapter II, Section 10.3]{Bredon}. 

\begin{multline*}
\xymatrix@C=0.5cm{
  0 \ar[r] & H^{0}_{c}(X,\mathcal{O}) \ar[rr] && H^{0}(X',\mathcal{O}) \ar[rr] && H^{0}(Z,i^{-1}\mathcal{O}) \ar[r] &\\ \ar[r] & H^{1}_{c}(X,\mathcal{O}) \ar[rr] && H^{1}(X',\mathcal{O}) \ar[r]  & \cdots }
\end{multline*}

Since $H^{0}_{c}(X,\mathcal{O})=0$, $H^{0}(X',\mathcal{O})=\CC$ and $H^{1}(X',\mathcal{O})=0$ it follows that $$H^{1}_{c}(X,\mathcal{O})\cong H^{0}(Z,i^{-1}\mathcal{O})/\CC.$$ 

Let $U\subset X'$ be a connected open neighbourhood of $Z$. Since $U$ is also a normal space and $Z$ is a thin set in $U$, by the criterion of connectedness \cite[Statement 13.8]{Grauert} we obtain that $U\setminus Z$ is a connected set. Thus the set $V=X'\setminus U$ is a compact subset in $X$, and $X\setminus V$ is connected.

A sequence of nested connected neighbourhoods $\{U_{n}\}_{n=1}^{\infty}$ of $Z$ with properties $\overline{U_{n+1}}\subset U_{n}$ and $\bigcap\limits_{n}U_{n}=Z$ induces a sequence of compact sets  $\{V_{n}\}_{n=1}^{\infty}$ in $X$ giving an exhaustion of $X$ such that $X\setminus V_{n}$ is connected.

Existence of such a sequence $\{U_{n}\}_{n=1}^{\infty}$ follows from metrizability of $X'$. Indeed, let $\rho$ be a metric compatible with the topology of $X'$. We define open sets as $$U_{n}:=\{x\in X'\mid \inf\limits_{z\in X'\setminus X}\rho(z,x)<\frac{1}{n}\}.$$

The sequence $\{U_{n}\}_{n=1}^{\infty}$ satisfies the properties $\overline{U_{n+1}}\subset U_{n}$ and $\bigcap\limits_{n}U_{n}=Z$. The proof of the corollary is complete.
\end{proof}

For example, let $X$ be a normal noncompact algebraic $G$-variety and $G$ be a complex linear algebraic group acting algebraically on $X$. By Sumihiro's results \cite[Theorem 3]{Sumi} there exists a $G$-equivariant compactification $X'$ of $X$(i.e. there exists a normal compact algebraic variety $X'$ on which an algebraic action of $G$ is given such that $X$ is embedded in $X'$ as an open subset and the algebraic action of $G$ on $X'$ is an extension of the given algebraic action of $G$ on $X$).

Additionally we assume that X is (b, 0)-compactifiable with $b \geq 1$. If $b>1$, then the analytic set $Z=X'\setminus X$ is not connected. But there exists a $G$-invariant Zariski open subvariety $X''\subset X'$ such that $X\subset X''$ and $X'\setminus X''$ is connected (i.e. $X''$ is (1,0)-compactifiable). In this case, if $X''$ admits the Hartogs phenomenon, then so does $X$. 

\section{Almost homogeneous $G$-varieties}
\label{s3}

Let $X$ be a complex analytic variety (i.e. reduced, irreducible complex analytic space) and $G$ be a connected complex Lie group acting holomorphically on $X$. In this situation $X$ is called a complex analytic \textit{$G$-variety}.

\begin{definition}
A complex analytic $G$-variety $X$ is called almost homogeneous if $X$ has an open $G$-orbit $\Omega$.
\end{definition}

Note that an open $G$-orbit $\Omega$ in an almost homogeneous complex analytic $G$-variety $X$ is unique and connected and $E:=X\setminus \Omega$ is a proper analytic subset \cite[Section 1.7, Proposition 4]{Akh}. 

The main object of this paper are normal (1,0)-compactifiable almost homogeneous complex algebraic $G$-varieties (see Definition \ref{maindef} for (1,0)-compactifiable varieties), where $G$ is a reductive complex Lie group. 

Recall that a connected complex Lie group $G$ is called reductive if $G$ has a compact real form $K$ (i.e. $K$ is a compact real Lie subgroup of $G$ and $Lie (K)\otimes\CC=Lie(G)$). 

This definition is closely related to the corresponding definition for connected linear algebraic groups. We recall that a connected linear algebraic group is called reductive if it is unipotent radical is trivial. A connected reductive linear algebraic group over $\CC$ is reductive in the sense of the first definition. Conversely, each connected reductive complex Lie group $G$ has a unique structure of a reductive linear algebraic group. Thus one can consider the subalgebra of regular functions $\CC[G]$ in the algebra of holomorphic functions $H^{0}(G,\mathcal{O})$.

Using Corollary \ref{Cor1}, we see that we have to study the space $H^{0}(Z,i^{-1}\mathcal{O})$. In the case where $G$ is a complex reductive Lie group we use the Harish-Chandra theorem (see section \ref{s4}). 

We conclude this section by proving the existence of a neighbourhood system of $Z:=X'\setminus X$ with a good property. 

Recall that $H^{0}(Z,i^{-1}\mathcal{O})=\varinjlim\limits_{U\supset Z}H^{0}(U,\mathcal{O})$, but the inductive limit may be taken over a cofinal system of neighbourhoods of the set $Z$.

Now let $X$ be an arbitrary complex analytic variety, $K$ be a compact Lie group acting holomorphically on $X$ and $Z$ be a connected compact $K$-invariant analytic set in $X$. 

\begin{proposition}\label{prop2}
There exists a cofinal system of neighbourhoods $\{U_{n}\}$ of $Z$ such that each $U_{n}$ is a $K$-invariant open neighbourhood.
\end{proposition}
\begin{proof}
Let us note that $X$ is metrizable. Let $\rho$ be a metric compatible with the topology of $X$. Let $\mu$ be a Haar measure on the compact Lie group $K$. Define $$\rho_{K}(x,y):=\int\limits_{K} \rho(k.x,k.y)d\mu(k)$$
Clearly, $\rho_K$ is a metric. Since $\mu$ is left-right-invariant, it follows that $$\rho_{K}(s.x,s.y)=\int\limits_{K} \rho((ks).x,(ks).y)d\mu(k)=\int\limits_{K} \rho(k.x,k.y)d\mu(k)=\rho_{K}(x,y).$$

Let $U$ be a relatively compact neighbourhood of $Z$. Since $\rho(k.x,k.y)$ is continuous on the compact set $K\times\overline{U}\times\overline{U}$, it follows that  $\rho_{K}(x,y)$ is continuous on $U\times U$ with respect to the original topology on $X$.

Note that the distance function $dist (x, Z):=\inf\limits_{z\in Z}\rho_{K}(x,z)$ is a continuous function on $U$. 

Let $U_{n}:=\{x\in U\mid dist (x, Z)<\frac{1}{n}\}$. It is an open neighbourhood of $Z$ with respect to the original topology. Note that for all $x\in U_{n}$ and for all $k\in K$ we have $$\inf\limits_{z\in Z}\rho_{K}(k.x,z)=\inf\limits_{z\in Z}\rho_{K}(x,k^{-1}.z)=\inf\limits_{z\in Z}\rho_{K}(x,z)<\frac{1}{n}.$$ 

It follows that $U_n$ is a $K$-invariant neighbourhood of $Z$. Now, let $V$ be an arbitrary neighbourhood of $Z$. Since there exists $n$ such that $\frac{1}{n}<\inf\limits_{z\in Z,x\in X\setminus V}\rho_{K}(x,z)$, it follows that $U_{n}\subset V$. This concludes the proof of Proposition \ref{prop2}. 
\end{proof}

Therefore $H^{0}(Z,i^{-1}\mathcal{O})=\mathop{\underrightarrow{\lim}}_{U_n}H^{0}(U_n,\mathcal{O})$, where the inductive limit is taken over the neighbourhood system $\{U_n\}$ of $Z$ where $U_{n}$ are $K$-invariant.

\section{The Harish-Chandra theorem and its applications}
\label{s4}
First we recall the necessary elements of the representation theory of the compact Lie groups (see \cite{Akh, Sepanski}).

Let $F$ be a Frechet space over $\CC$, let $K$ be a compact Lie group and $$\rho\colon K\to Aut(F)$$ be a continuous representation (i.e. $\rho$ is a homomorphism of $K$ into the group of invertible continuous linear operators on $F$ such that $K\to F, k\mapsto \rho(k)f$ is continuous for every vector $f\in F$). 

We define the following $K$-invariant subspaces in $F$. 
\begin{definition}\noindent
\begin{itemize}
\item Define the subspace of all $K$-finite vectors as $$F^{0}:=\{f\in F\mid \dim span\langle\rho(k)f\mid k\in K\rangle<\infty\}$$
\item Define the subspace of all $K$-differentiable vectors as $$F^{\infty}:=\{f\in F\mid K\to F, k\mapsto \rho(k)f \text{ is differentiable}\}$$
\end{itemize}
\end{definition}

Let $\widehat{K}$ be the set of all equivalence classes of finite-dimensional irreducible linear representations of $K$.
For each $\delta\in \widehat{K}$ there is a continuous linear operator $$E_{\delta}\colon F\to F, E_{\delta}(f):=\int\limits_{K}\overline{\chi_{\delta}(k)}\rho(k)fd\mu(k)$$ where $\chi_{\delta}$ is the character of a representation in $\delta$. 

We have the following properties of the operators $E_{\delta}$ \cite[Section 5.1, Theorem 3]{Akh}:
\begin{enumerate}
\item Each $E_{\delta}$ is a continuous projection;
\item $E_{\delta}E_{\delta'}=0$ if $\delta\neq\delta'$;
\item $\rho(k)E_{\delta}=E_{\delta}\rho(k)$ for all $k\in K$
\end{enumerate}

Let $F_{\delta}:=E_{\delta}(F)$. Note that $F_{\delta}$ is a closed $K$-invariant subspace of $F$ for each $\delta\in\widehat{K}$ and $F^{0}=\bigoplus\limits_{\delta\in\widehat{K}}F_{\delta}$. Moreover $F_{\delta}$ is the isotypic component of $F$ of type $\delta$ (i.e. $F_{\delta}$ consists of those vectors in $F$, whose $K$-orbit is contained in a finite-dimensional $K$-submodule isomorphic to $mV_{\delta}$) \cite[Section 5.1, Theorem 4]{Akh}.

We have the series representation for any differentiable vector; this series converges absolutly with respect to any continuous seminorm on $F$. It is the so-called Harish-Chandra theorem \cite[Section 5.1, Theorem 5]{Akh}. 

\begin{theorem}\label{HarCh}
For any vector $f\in F^{\infty}$ $$f=\sum\limits_{\delta\in\widehat{K}} E_{\delta}f$$ where the convergence is absolute with respect to any continuous seminorm on $F$.
\end{theorem}

Now let $X$ be a complex analytic $K$-variety and $K$ be a connected compact Lie group acts holomorphically on $X$. Then $K$ has a linear representation on the Frechet space $F=H^{0}(X,\mathcal{O})$. In this case, we have $F=F^{\infty}$ \cite[Section 5.2, Proposition]{Akh}. 

Let $G$ be a connected reductive Lie group with a real compact form $K$. Let $\Omega$ be a complex algebraic homogeneous $G$-variety and $\CC[\Omega]$ be the algebra of regular function on $\Omega$. Let $W\subset \Omega$ be a $K$-invariant domain.

\begin{proposition}\label{HarCh2}
For every $f\in H^{0}(W,\mathcal{O})$  we have $$f=\sum\limits_{\delta\in\widehat{K}}E_{\delta}f$$ where the convergence is absolute with respect to any continuous seminorm on $H^{0}(W,\mathcal{O})$ and moreover $E_{\delta}f\in \CC[\Omega]$.
\end{proposition}
\begin{proof}
Using \cite[Section 5.3, Theorem 2]{Akh}, we conclude the proof of the Proposition \ref{HarCh2}. 
\end{proof}

In other words, the algebra $\CC[\Omega]$ is dense in $H^{0}(W,\mathcal{O})$.

Let $G$ and $K$ be as above, and let $X$ be a normal (1,0)-compactifiable almost homogeneous complex algebraic $G$-variety. Let $X'$ be a corresponding $G$-equi\-variant compactification of $X$ and $Z:=X'\setminus X$. Denote by $\mathcal{G}(X')$ the set of $G$-stable prime divisors on $X'$, let $$Y:=X'\setminus \bigcup\limits_{D\in\mathcal{G}(X'),D\subset X} D,$$

It is a Zariski open algebraic subvariety in $X'$ which is a normal almost homogeneous complex algebraic $G$-variety. Denote by $\CC[Y]$ the algebra of regular functions on $Y$. 

Note that $\{D\in\mathcal{G}(X'),D\subset X\}$ may be an empty set, in which case we obtain $Y=X'$ and $\CC[Y]=\CC$.

The vector space $H^{0}(Z,i^{-1}\mathcal{O})$ has a natural topological vector space structure (namely, it has the direct limit topology).

\begin{lemma}\label{Prop5}
The canonical homomorphism $$\CC[Y]\to H^{0}(Z,i^{-1}\mathcal{O})$$ is injective and with a dense range. 
\end{lemma}
\begin{proof}
Using Proposition \ref{prop2}, we obtain that there exists a cofinal system of neighbourhoods $\{U_n\}$ of $Z$ such that each $U_{n}$ is $K$-invariant. 

Consider the Frechet space $H^{0}(U_{n},\mathcal{O})$. Using Theorem \ref{HarCh}, we obtain any $f\in H^{0}(U_{n},\mathcal{O})$ has a series representation $$f=\sum\limits_{\delta\in\widehat{K}} E_{\delta}f$$ with the uniform convergence on compact sets in $U_n$.
Using Proposition \ref{HarCh2}, we obtain $E_{\delta}f\in \CC[\Omega]$. 

Since $f\in H^{0}(U_{n},\mathcal{O})$ and $U_{n}$ intersects with each prime $G$-stable divisor on $Y$, it follows that the rational function $E_{\delta}(f)$ has no poles on each prime $G$-stable divisor on $Y$. Thus we obtain $E_{\delta}(f)\in \CC[Y]$. Moreover $\CC[Y]$ is the dense subspace in $H^{0}(U_{n},\mathcal{O})$.

Clearly, the canonical homomorphism $\CC[Y]\to H^{0}(Z,i^{-1}\mathcal{O})$ is injective. 

Now let $[f]\in H^{0}(Z,i^{-1}\mathcal{O})$. It is represented by $f$ in $H^{0}(U_{n},\mathcal{O})$ for some $n$. Since $\CC[Y]$ is the dense subspace in $H^{0}(U_{n},\mathcal{O})$, it follows that there exists $\{f_{n}\}_{n\in\mathbb{N}}\subset \CC[Y]$ such that $\lim\limits_{n\to\infty}f_{n}=f$ in $H^{0}(U_{n},\mathcal{O})$. Continuity of the canonical map $$H^{0}(U_{n},\mathcal{O})\to H^{0}(Z,i^{-1}\mathcal{O}), f\mapsto [f]$$ implies that $\lim\limits_{n\to\infty}[f_{n}]=[f]$. Thus the map $\CC[Y]\to H^{0}(Z,i^{-1}\mathcal{O})$ has a dense range.  The proof of the lemma is complete.
\end{proof}

Using Corollary \ref{Cor1} and Lemma \ref{Prop5}, we obtain the following proposition.

\begin{proposition}\label{Prop6}
Let $G$ be a connected complex reductive Lie group, and let $X$ be a normal (1,0)-compactifiable almost homogeneous algebraic $G$-variety. Then $X$ admits the Hartogs phenomenon if and only if $\CC[Y]=\CC$.
\end{proposition}

\section{Weight criterion for the Hartogs phenomenon}
\label{s5}

Let $X, X', Y$ be as in Section \ref{s4}. We have to investigate $\CC[Y]$. Since $G$ acts algebraically on $Y$ it follows that $\CC[Y]$ is a representation of $G$. 

We recall some facts of the representation theory of reductive groups (see \cite{Brion, Humphreys}). 

Let $G$ be a connected reductive complex Lie group, $B\subset G$ be a Borel subgroup and $T\subset B$ be a maximal algebraic torus with the character lattice $\XX(T)$. For the Lie algebra $\mathfrak{g}=Lie(G)$ we have the root decomposition $$\mathfrak{g}=\mathfrak{t}\oplus\bigoplus\limits_{\alpha\in R}\mathfrak{g}_{\alpha}, $$ with respect to the adjoint reprepresentation, where $R\subset \mathcal{X}(T)$ is the root system of $G$. 

Let $R_{+}\subset R$ be the set of positive roots. This set is uniquely defined by the following condition $$Lie(B)=\mathfrak{t}\oplus\bigoplus\limits_{\alpha\in R_{+}}\mathfrak{g}_{\alpha}.$$ 

Denote by $S$ the set of simple roots and denote by $S^{\vee}$ the set of dual simple roots (it is a subset in the 1-parameter subgroup lattice $\XX^{*}(T)$). 

Denote by $C_{+}$ the dominant Weyl chamber. Recall that $$C_{+}:=\{t\in \XX(T)\otimes\mathbb{R}\mid \langle t,\alpha^{\vee}\rangle\geq 0, \forall \alpha\in S\}.$$ 

Define the set of dominant characters $\XX_{+}(T)$ as $\XX(T)\cap C_{+}$. 

We need the following facts. 
\begin{enumerate}
\item There is a bijective correspondence between the set $\XX_{+}(T)$ of dominant characters and the set $\widehat{G}$ of isomorphism classes of irreducible representations of $G$.
\item For any rational representation $V$ of the reductive group $G$ (i.e. a representation such that every $v\in V$ is contained in a finite-dimensional $G$-stable subspace on which $G$ acts algebraically) we have the canonical decomposition 
$$V\cong\bigoplus\limits_{\lambda\in\XX_{+}(T)}V_{\lambda}^{(B)}\otimes V(\lambda)$$ where $V_{\lambda}^{(B)}:=\{v\in V\mid b.v=\lambda(b)v\}$ is the set of $B$-eigenvectors of weight $\lambda$. 
\end{enumerate}

Since $\CC[Y]$ is a rational representation of $G$ \cite[Lemma 1.5]{Brion}, it follows that the canonical decomposition for $\CC[Y]$ has the form

$$\CC[Y]\cong\bigoplus\limits_{\lambda\in\XX_{+}(T)}\CC[Y]_{\lambda}^{(B)}\otimes V(\lambda).$$

Define the weight monoid of $Y$ as $\Lambda_{+}(Y):=\{\lambda\in \XX_{+}(T)\mid \CC[Y]_{\lambda}^{(B)}\neq 0\}$.

We obtain the following weight criterion for the Hartogs phenomenon.

\begin{theorem}
Let $G$ be a connected complex reductive Lie group, and let $X$ be a normal (1,0)-compactifiable almost homogeneous algebraic $G$-variety. Then $X$ admits the Hartogs phenomenon if and only if $\Lambda_{+}(Y)=0$ and $\CC[Y]^{B}=\CC$.
\end{theorem}

\section{Convex geometric criterion for the Hartogs phenomenon in spherical varieties}
\label{s6}

In this section we recall some facts from the theory of spherical varieties (see \cite{Gandini, Tim}). 

Let $G$ be a complex reductive Lie group, $B\subset G$ be a Borel subgroup, $T\subset B$ be a maximal algebraic torus.

\begin{definition}\label{defspher}
A normal almost homogeneous complex algebraic $G$-variety with the open orbit $\Omega$ is called spherical if $\Omega$ contains an open $B$-orbit $O$.
\end{definition}

Note that the open $B$-orbit $O$ is an affine variety \cite[Theorem 3.5]{Tim}.

Also, a normal algebraic $G$-variety $X$ is spherical if and only if any $B$-invariant rational function on $X$ is constant (i.e. $\mathbb{C}(X)^{B}=\mathbb{C}$) \cite[Theorem 2.8]{Gandini} and if and only if $X$ contains finitely many $B$-orbits \cite[Theorem 2.11]{Gandini}. 

With the open orbit $\Omega$ we associate the weight lattice 

\begin{equation*}
M:=\{\lambda\in \mathfrak{X}(T)\mid \mathbb{C}(\Omega)^{(B)}_{\lambda}\neq0\}.
\end{equation*}
and the dual weight lattice $N:=Hom(M,\mathbb{Z})$.

Also, let $M_{\mathbb{R}}:=M\otimes\mathbb{R}$ and $N_{\mathbb{R}}:=N\otimes\mathbb{R}$. Note that $M$ is a sublattice in the character lattice $\mathfrak{X}(T)$ of the torus $T$. Recall that $B=TU$ (here $U$ is the unipotent radical of $B$), and since $U$ has no nontrivial characters we obtain $\XX(B)=\XX(T)$. We define $$\mathbb{C}(\Omega)^{(B)}:=\{f\in \mathbb{C}(\Omega)\setminus\{0\}\mid \exists\lambda\in\XX(T): b.f=\lambda(b)f, \forall b\in B\}.$$ 

It follows that we have the isomorphism $$M\cong\mathbb{C}(\Omega)^{(B)}/\mathbb{C}^{*}$$ 
given by $\lambda\mapsto [f]$ where $[f]$ is the equivalence class of any $$f\in \mathbb{C}(\Omega)^{(B)}_{\lambda}=\{f\in \mathbb{C}(\Omega)\mid b.f=\lambda(b)f, \forall b\in B\}.$$

Let $Y$ be as in section \ref{s4}. Denote by $\mathcal{B}(Y)$ the set of all prime $B$-stable divisors on $Y$. Note that $\mathcal{B}(Y)$ is the union of the set of all prime $G$-stable divisors on $Y$ and the set of all prime $B$-stable but not $G$-stable divisors on $Y$. 

Each $B$-stable divisor $D\in \mathcal{B}(Y)$ defines the discrete valuation $$v_{D}\colon \mathbb{C}(\Omega)\setminus\{0\}\to \mathbb{Z}.$$ 

Recall that $v_{D}(f)$ is the order of zeros or poles of $f$ at $D$. Also, the valuation $v_{D}$ defines a point in the dual weight lattice $a_{D}\in N=Hom(M,\mathbb{Z})$ by setting $\langle a_{D},\lambda\rangle:=v_{D}(f)$ for $f\in \mathbb{C}(\Omega)^{(B)}_{\lambda}$.

Define $$L:=\{\lambda\in M_\mathbb{R}\mid \langle a_{D},\lambda\rangle \geq 0, \forall D\in\mathcal{B}(Y)\}$$ and we identify $L$ with its image with respect to the injective map $$\iota\colon M_{\mathbb{R}}\hookrightarrow \mathfrak{X}(T)\otimes\mathbb{R}.$$ 

We obtain the following description of the set of $B$-eigenvectors of $\CC[Y]$.  

\begin{lemma}
$\mathbb{C}[Y]^{(B)}_{\lambda}=\begin{cases}
\mathbb{C}[O]^{(B)}_{\lambda}\neq 0, & \text{if $\lambda\in L\cap M$;} \\
0, & \text{if $\lambda\not\in L\cap M$.}\end{cases}$

In particular, we obtain the description of the weight monoid of $Y$ $$\Lambda_{+}(Y)=L\cap M$$
\end{lemma}

\begin{proof}
Note that $$\mathbb{C}[Y]^{(B)}_{\lambda}=\{f\in \mathbb{C}[O]^{(B)}_{\lambda}\mid \langle a_{D},\lambda\rangle=v_{D}(f)\geq 0, \forall D\in\mathcal{B}(Y)\}.$$ 

If $\lambda\in L$, then for any $f\in \mathbb{C}[O]^{(B)}_{\lambda}$ we obtain $\langle a_{D}, \lambda\rangle\geq 0$ for all $D\in \mathcal{B}(Y)$. It follows that $\mathbb{C}[Y]^{(B)}_{\lambda}=\mathbb{C}[O]^{(B)}_{\lambda}$. 

By definition of $M$, we have $\mathbb{C}(\Omega)^{(B)}_{\lambda}\neq 0$ for all $\lambda\in M$. Since $$\mathbb{C}[Y]^{(B)}_{\lambda}=\mathbb{C}(\Omega)^{(B)}_{\lambda}$$ for $\lambda\in L$, it follows that $\mathbb{C}[Y]^{(B)}_{\lambda}\neq 0$.

Now let $\lambda\not\in L$ and we assume there exits $f\in \mathbb{C}[Y]^{(B)}_{\lambda}$ such that $f\neq 0$. Then there is $D\in\mathcal{B}(Y)$ such that $v_{D}(f)<0$. This is a contradiction; the proof of the lemma is complete.
\end{proof}

We obtain the following corollary.

\begin{corollary}\label{cormain11}
Let $X$ be a (1,0)-compactifiable spherical $G$-variety. Then $X$ admits the Hartogs phenomenon if and only if $L=0$.
\end{corollary}

Now we reformulate this criterion in terms of the dual weight lattice $N$.

Define a cone $C$ in the space $N_{\mathbb{R}}$ as $$C:=\mathbb{R}_{\geq 0}\big\langle a_{D}\mid D\in \mathcal{B}(Y)\big\rangle.$$ 

We obtain the following criterion.

\begin{corollary}\label{cormain12}
Let $X$ be a (1,0)-compactifiable spherical $G$-variety. Then $X$ admits the Hartogs phenomenon if and only if $C=N_{\mathbb{R}}$.
\end{corollary}

\begin{proof}
Since $C^{\vee}=L$, where $C^{\vee}$ is the dual cone of $C$, it follows that $C=N_{\mathbb{R}}$ if and only if $L=0$. The proof of the corollary is complete.
\end{proof}

This criterion can be also formulated in terms of colored fans. 

First we recall the definition of valuation cone (for more details see \cite[Sections 4, 10]{Gandini} or \cite[Chapter 4 and Appendix B]{Tim}). 

Let $v\colon\mathbb{C}(\Omega)\setminus\{0\}\to \mathbb{Q}$ be a discrete $\mathbb{Q}$-valued valuation. The valuation $v$ is called \textit{$G$-invariant} if $v(g.f)=v(f)$ for all $f\in\mathbb{C}(\Omega)$ and for all $g\in G$. The set of $G$-invariant valuations of $\mathbb{C}(\Omega)$ is denoted by $\mathcal{V}(\Omega)$.

Note that a $G$-invariant valuation $v$ defines a point $a_{v}\in N_{\mathbb{Q}}=N\otimes\mathbb{Q}$ by setting $$\langle a_{v},\lambda\rangle:=v(f)$$ for some $f\in \mathbb{C}(\Omega)^{(B)}_{\lambda}$. By \cite[Corollary 4.9]{Gandini} we will regard $\mathcal{V}(\Omega)$ as a subset of $N_{\mathbb{Q}}$. Note that it is a finitely generated convex rational cone of maximal dimension \cite[Corollary 10.6]{Gandini} and it is called \textit{valuation cone of $\Omega$}. By $\mathcal{V}_{\mathbb{R}}(\Omega)$ we denote the cone generated by the set $\mathcal{V}(\Omega)$ in $N_{\mathbb{R}}$.

Now by $\mathcal{B}(\Omega)$ we denote the set of all prime $B$-stable divisors of $\Omega$. For each $D\in\mathcal{B}(\Omega)$ denote by $\overline{D}$ the closure of $D$ in $Y$, where $Y$ is as in section \ref{s4}. Since $v_{D}=v_{\overline{D}}$ (after the identification $\mathbb{C}(\Omega)=\mathbb{C}(Y)$), it follows that we may identify the set $\mathcal{B}(\Omega)$ with the subset $\{\overline{D}\mid D\in \mathcal{B}(\Omega)\}$ of $\mathcal{B}(Y)$. 

We recall the definition of colored cone and colored fan. 

\begin{definition}\indent
\begin{itemize}
\item A colored cone is a pair $(\sigma,\mathcal{F})$ with $\sigma\subset N_{\mathbb{R}}$ and $\mathcal{F}\subset \mathcal{B}(\Omega)$ with the following properties: 
\begin{itemize}
\item $\sigma$ is a convex cone generated by $\{a_{D}\mid D\in\mathcal{F}\}$ and finitely many elements of $\mathcal{V}(\Omega)$;
\item the relative interior of $\sigma$ intersects $\mathcal{V}(\Omega)$ non trivially;
\item $\sigma$ contains no lines and $0\not\in \{a_{D}\mid D\in\mathcal{F}\}$. 
\end{itemize}
\item A colored face of a colored cone $(\sigma,\mathcal{F})$ is a pair $(\sigma',\mathcal{F}')$ such that $\sigma'$ is a face of $\sigma$, the relative interior of $\sigma'$ intersects non trivially $\mathcal{V}(\Omega)$ and $\mathcal{F}'=\{D\in\mathcal{F}\mid a_{D}\in \sigma'\}$;
\item A colored fan is a finite set $\Sigma$ of colored cones with the following properties:
\begin{itemize}
\item every colored face of a colored cone of $\Sigma$ is in $\Sigma$;
\item for any $v\in\mathcal{V}(\Omega)$, there exists at most one cone $(\sigma,\mathcal{F})$ such that $v$ is in the relative interior of $\sigma$.
\end{itemize}
\item The support of a colored fan $\Sigma$ is the set $|\Sigma|:=\bigcup\limits_{(\sigma,\mathcal{F})\in\Sigma}\sigma$.
\end{itemize}
\end{definition}

Let $X$ be a spherical $G$-variety with the open orbit $\Omega$. For any $G$-orbit $Y$ of $X$, denote by $X_Y$ the $G$-stable subset $\{x\in X\mid \overline{G.x}\supset Y\}$ (it is a $G$-stable open subvariety of $X$ containing a unique closed $G$-orbit $Y$). Let $D_{1},\cdots, D_{m}$ be the $G$-stable divisors of $X_Y$. Let $\mathcal{F}$ be the set of $B$-stable but not $G$-stable divisors $D$ in $X$ which contain the closed orbit of $X_Y$. Let $\sigma$ be the cone of $N_{\mathbb{R}}$ generated by points $a_{v}$ for $v\in\mathcal{F}$ and points $a_{D_i}\in N$ for $i\in\{1,\cdots, m\}$. Thus we obtain a colored cone $(\sigma,\mathcal{F})$. Moreover, the set of colored cones constructed in this way forms a colored fan $\Sigma_{X}$.

We have the following Luna--Vust theorem on the classification of $\Omega$-embeddings.

\begin{theorem}
The map $X\to\Sigma_{X}$ is a bijection between the set of isomorphism classes of spherical $G$-varieties (with the same open orbit $\Omega$) and the set of colored fans.
\end{theorem}

Let us now give few properties of spherical varieties via the classification. 
\begin{remark}\label{propfanvar}
Let $X_{\Sigma}$ be a spherical variety with the open orbit $\Omega$ and with the colored fan $\Sigma$.  
\begin{enumerate}
\item $X_{\Sigma}$ is compact if and only if $\Sigma$ is complete (i.e. $|\Sigma|\supset\mathcal{V}_{\mathbb{R}}(\Omega)$); 
\item There is a bijective correspondence between the $G$-orbits in $X_{\Sigma}$ and the colored cones in $\Sigma$. 
\item If $W_{1},W_{2}$ are $G$-orbits with corresponding colored cones $(\sigma_{1},\mathcal{F}_1)$, $(\sigma_{2},\mathcal{F}_2)$ then $W_{1}\subset \overline{W_{2}}$ if and only if $(\sigma_{2},\mathcal{F}_2)$ is a colored face of the colored cone $(\sigma_{1},\mathcal{F}_1)$.
\end{enumerate}
\end{remark}

Now we prove the following lemma.

\begin{lemma}\label{lemconnect}
Let $X_{\Sigma}$ be a noncompact spherical variety with a colored fan $\Sigma$. Then $X_{\Sigma}$ is (1,0)-compactifiable if and only if the open set $\mathcal{V}_{\mathbb{R}}(\Omega)\setminus |\Sigma|$ is connected.
\end{lemma}
\begin{proof}
Let $X_{\Sigma'}$ be a compactification of $X_{\Sigma}$. In this case we have $\Sigma\subset \Sigma'$.

First, using \cite[Corollaire 1]{Brion1} and the GAGA principle we have $H^{1}(X_{\Sigma'},\mathcal{O})=0$ for a compact spherical variety $X_{\Sigma'}$. So, we have to prove that $Z:=X_{\Sigma'}\setminus X_{\Sigma}$ is connected if and only if $\mathcal{V}_{\mathbb{R}}(\Omega)\setminus |\Sigma|$ is connected.

Now by $O(\sigma,\mathcal{F})$ we denote the $G$-orbit corresponding to the colored cone $(\sigma,\mathcal{F})$. Since the set $Z=X_{\Sigma'}\setminus X_{\Sigma}$ is $G$-stable, it is a union of $G$-orbits. Namely, we have 
$$Z=\bigcup\limits_{(\sigma,\mathcal{F})\in \Sigma'\setminus\Sigma}O(\sigma,\mathcal{F}).$$

Further, $Z$ is connected if and only if for any two $G$-orbits $V=O(\sigma,\mathcal{F}), V'=O(\sigma',\mathcal{F}')$ such that $(\sigma,\mathcal{F}), (\sigma',\mathcal{F}')\in\Sigma'\setminus\Sigma$ there exists a sequence of $G$-orbits $\{V_{i}=O(\sigma_i,\mathcal{F}_i)\}_{i=0}^{n}$ with the following properties:
\begin{enumerate}
\item $V_0=V, V_{n}=V'$;
\item $(\sigma_i,\mathcal{F}_i)\in\Sigma'\setminus\Sigma$;
\item $\overline{V_{i}}\cap\overline{V_{i+1}}\neq\emptyset$.
\end{enumerate}

But this is equivalent to the fact that for any two colored cones $(\sigma,\mathcal{F}), (\sigma',\mathcal{F}')\in\Sigma'\setminus\Sigma$ there exists a sequence of colored cones $\{C_{i}=(\sigma_i,\mathcal{F}_i)\}_{i=0}^{n}$ with the following properties:

\begin{enumerate}
\item $C_0=(\sigma,\mathcal{F}), C_n=(\sigma',\mathcal{F}')$;
\item $C_i=(\sigma_i,\mathcal{F}_i)\in\Sigma'\setminus\Sigma$;
\item For cones $\sigma_{i}$ and $\sigma_{i+1}$ there exists a colored cone $(\sigma_{i,i+1},\mathcal{F}_{i,i+1})\in\Sigma'\setminus\Sigma$ such that $\sigma_{i}$ and $\sigma_{i+1}$ are its faces. 
\end{enumerate}

This is equivalent to $\mathcal{V}_{\mathbb{R}}(\Omega)\setminus |\Sigma|$ being a connected set. The proof of the lemma is complete.
\end{proof}

We recall that $$Y=X_{\Sigma'}\setminus \bigcup\limits_{D\in\mathcal{G}(X_{\Sigma'}),D\subset X_{\Sigma}} D$$ where $\mathcal{G}(X_{\Sigma'})$ is the set of all prime $G$-stable divisors on $X_{\Sigma'}$.

Let $\mathcal{OG}_{k}(X_{\Sigma'})$ be the set of all $k$-codimensional $G$-orbit on $X_{\Sigma'}$, and $$\mathcal{OG}(X_{\Sigma'}):=\bigcup\limits_{k=1}^{\dim X_{\Sigma'}}\mathcal{OG}_{k}(X_{\Sigma'}).$$

Now we consider the spherical variety $$\widehat{Y}:=X_{\Sigma'}\setminus\bigcup\limits_{O\in\mathcal{OG}(X_{\Sigma'}),\overline{O}\subset X_{\Sigma}}O,$$ where $\overline{O}$ is the closure of $O$ in $X_{\Sigma'}$. Since $\mathcal{G}(X_{\Sigma'})=\{\overline{O}\mid O\in \mathcal{OG}_{1}(X_{\Sigma'})\}$, it follows that $\widehat{Y}\subset Y$ and $codim(Y\setminus\widehat{Y})>1$. Therefore $\mathbb{C}[Y]=\mathbb{C}[\widehat{Y}]$, $\mathcal{B}(Y)=\mathcal{B}(\widehat{Y})$. 

The colored fan $\Sigma_{\widehat{Y}}$ of $\widehat{Y}$ is obtained as follows.
 
\begin{lemma}\label{lemmaonY}
$$\Sigma_{\widehat{Y}}=\{(\tau,\mathcal{F}')\in \Sigma'\mid (\tau,\mathcal{F}') \text{ is a colored face of a colored cone in } \Sigma'\setminus\Sigma \}.$$
\end{lemma}

\begin{proof}
Consider a colored cone $(\tau,\mathcal{F}')\in\Sigma'$ and let $O=O(\tau,\mathcal{F}')$ be the corresponding $G$-orbit. Recall that the set $Z=X_{\Sigma'}\setminus X_{\Sigma}$ is a union of $G$-orbits, namely: $$Z=\bigcup\limits_{(\sigma,\mathcal{F})\in \Sigma'\setminus\Sigma}O(\sigma,\mathcal{F}).$$

We obtain the following obvious statements: $(\tau,\mathcal{F}')\in \Sigma_{\widehat{Y}}$ if and only if $\overline{O}\cap Z\neq\emptyset $ if and only if there exists a colored cone $(\sigma,\mathcal{F})\in\Sigma'\setminus\Sigma$ such that $(\tau,\mathcal{F}')$ is a colored face of  $(\sigma,\mathcal{F})$. The proof of the lemma is complete.
\end{proof} 

\begin{lemma} $|\Sigma_{\widehat{Y}}|\cap\mathcal{V}_{\mathbb{R}}(\Omega)=\overline{\mathcal{V}_{\mathbb{R}}(\Omega)\setminus |\Sigma|}.$
\end{lemma}

\begin{proof}
By Lemma \ref{lemmaonY} we obtain that $$|\Sigma_{\widehat{Y}}|=\bigcup\limits_{(\sigma,\mathcal{F})\in\Sigma'\setminus\Sigma}\sigma.$$

A point $p\in\mathcal{V}_{\mathbb{R}}(\Omega)$ belongs to $\overline{\mathcal{V}_{\mathbb{R}}(\Omega)\setminus |\Sigma|}$  if and only if there exists a sequence $\{p_n\}$ of points $p_n\in \mathcal{V}_{\mathbb{R}}(\Omega)\setminus |\Sigma|$ such that $p_{n}\to p$ as $n\to\infty$. We can assume that every $p_n$ belongs to $\sigma\cap\mathcal{V}_{\mathbb{R}}(\Omega)$ for a colored cone $(\sigma,\mathcal{F})\in \Sigma'\setminus\Sigma$. We obtain that $p_{n}\in |\Sigma_{\widehat{Y}}|\cap\mathcal{V}_{\mathbb{R}}(\Omega)$. Since $|\Sigma_{\widehat{Y}}|\cap\mathcal{V}_{\mathbb{R}}(\Omega)$ is a closed set, it follows that $p\in |\Sigma_{\widehat{Y}}|\cap\mathcal{V}_{\mathbb{R}}(\Omega)$. The proof of the lemma is complete.
\end{proof}

We recall that $$C=\mathbb{R}_{\geq 0}\big\langle a_{D}\mid  D\in \mathcal{B}(Y)\big\rangle\subset N_{\RR}.$$

Now we prove the following lemma.

\begin{lemma}\label{lemmaonC}
$C=\mathbb{R}_{\geq 0}\big\langle\overline{\mathcal{V}_{\mathbb{R}}(\Omega)\setminus |\Sigma|}\cup \{a_{D}\mid D\in \mathcal{B}(\Omega)\}\big\rangle. $

\end{lemma}
\begin{proof}

Note that $\mathcal{B}(\widehat{Y})$ is the union of the set $\mathcal{G}(\widehat{Y})$ of prime $G$-stable divisors of $\widehat{Y}$ and the set $\{\overline{D}\mid D\in\mathcal{B}(\Omega)\}$, where $\overline{D}$ is the closure in $\widehat{Y}$ of a prime $B$-stable divisor $D\subset \Omega$. 

By $\mathcal{F}(\widehat{Y})$ we denote the union of all subsets $\mathcal{F}\subset\mathcal{B}(\Omega)$ such that $(\sigma,\mathcal{F})\in \Sigma_{\widehat{Y}}$. 

Thus we have the following equalities:
\begin{multline*}
C=\mathbb{R}_{\geq 0}\big\langle a_{D}\mid D\in \mathcal{B}(Y)\big\rangle=\mathbb{R}_{\geq 0}\big\langle a_{D}\mid D\in \mathcal{B}(\widehat{Y})\big\rangle=\\=\mathbb{R}_{\geq 0}\big\langle \{a_{D}\mid D\in \mathcal{G}(\widehat{Y})\cup\mathcal{F}(\widehat{Y})\}\cup\{a_{D}\mid D\in\mathcal{B}(\Omega)\}\big\rangle.
\end{multline*}

Further we have the decomposition $$|\Sigma_{\widehat{Y}}|=(|\Sigma_{\widehat{Y}}|\cap\mathcal{V}_{\mathbb{R}}(\Omega))\cup \overline{|\Sigma_{\widehat{Y}}|\setminus\mathcal{V}_{\mathbb{R}}(\Omega)}=\overline{\mathcal{V}_{\mathbb{R}}(\Omega)\setminus |\Sigma|}\cup \overline{|\Sigma_{\widehat{Y}}|\setminus\mathcal{V}_{\mathbb{R}}(\Omega)}.$$

Since the set of generators of all colored cones of $\Sigma_{\widehat{Y}}$ is the set $$\{a_{D}\mid D\in \mathcal{G}(\widehat{Y})\cup\mathcal{F}(\widehat{Y})\},$$  it follows that $$\mathbb{R}_{\geq 0}\big\langle a_{D}\mid  D\in \mathcal{G}(\widehat{Y})\cup\mathcal{F}(\widehat{Y})\big\rangle=\mathbb{R}_{\geq 0}\big\langle|\Sigma_{\widehat{Y}}|\big\rangle=\mathbb{R}_{\geq 0}\big\langle\overline{\mathcal{V}_{\mathbb{R}}(\Omega)\setminus |\Sigma|}\cup \overline{|\Sigma_{\widehat{Y}}|\setminus\mathcal{V}_{\mathbb{R}}(\Omega)}\big\rangle.$$

Now if a ray $\mathbb{R}_{\geq0}\langle a_{D}\rangle$ intersects with $N_{\mathbb{R}}\setminus\mathcal{V}(\Omega)$ then $D$ is not $G$-stable and thus $D\in\mathcal{F}(\widehat{Y})\subset\mathcal{B}(\Omega)$.

We obtain the following equalities:
\begin{multline*}
C=\mathbb{R}_{\geq 0}\big\langle \{a_{D}\mid D\in \mathcal{G}(\widehat{Y})\cup\mathcal{F}(\widehat{Y})\}\cup\{a_{D}\mid D\in\mathcal{B}(\Omega)\}\big\rangle=\\=
\mathbb{R}_{\geq 0}\big\langle\overline{\mathcal{V}_{\mathbb{R}}(\Omega)\setminus |\Sigma|}\cup \overline{|\Sigma_{\widehat{Y}}|\setminus\mathcal{V}_{\mathbb{R}}(\Omega)}\cup\{a_{D}\mid D\in\mathcal{B}(\Omega)\}\big\rangle=\\=\mathbb{R}_{\geq 0}\big\langle\overline{\mathcal{V}_{\mathbb{R}}(\Omega)\setminus |\Sigma|}\cup \{a_{D}\mid D\in \mathcal{B}(\Omega)\}\big\rangle.
\end{multline*}
The proof of the lemma is complete.
\end{proof}

Thus we obtain the following convex geometric criterion. 

\begin{theorem}\label{ThmC}
Let $X_{\Sigma}$ be a noncompact spherical $G$-variety with the open $G$-orbit $\Omega$ and with the colored fan $\Sigma$ such that $\mathcal{V}_{\mathbb{R}}(\Omega)\setminus |\Sigma|$ is a connected set. Then $X_{\Sigma}$ admits the Hartogs phenomenon if and only if $$\mathbb{R}_{\geq 0}\langle\overline{\mathcal{V}_{\mathbb{R}}(\Omega)\setminus |\Sigma|}\cup \{a_{D}\mid D\in \mathcal{B}(\Omega)\}\rangle=N_\mathbb{R}.$$
\end{theorem} 

\begin{remark}\label{remspher}
Let $\Omega$ be a noncompact spherical homogeneous $G$-variety. The colored fan of $\Omega$ is $\Sigma=\{(0,\emptyset)\}$. The formula $$\mathbb{R}_{\geq 0}\langle\mathcal{V}_{\mathbb{R}}(\Omega)\cup \{a_{D}\mid D\in \mathcal{B}(\Omega)\}\rangle=N_\mathbb{R}$$ is true for every $\Omega$, because in this case $Y=X_{\Sigma'}$ and $\mathbb{C}[Y]=\mathbb{C}$. Note that the connectedness condition for $\mathcal{V}_{\mathbb{R}}(\Omega)\setminus \{0\}$ is automatically satisfied, except for the case $\mathcal{V}_{\mathbb{R}}(\Omega)=N_{\mathbb{R}}$ and $rk(N)=1$ (i.e. the case of horospherical
homogeneous $G$-varieties of rank one). In this situation, $\Omega$ admits the Hartogs phenomenon. For the case of horospherical
homogeneous $G$-varieties of rank one see Remark \ref{remhorospher} in the next section. 
\end{remark} 

\subsection{The case of horospherical varieties}
\label{s6.1}

First, we recall some notions and notation.

Let $G$ be a complex reductive Lie group, $B\subset G$ be a Borel subgroup, $T\subset B$ be a maximal algebraic torus, $U$ be the unipotent radical of $B$ (it is a maximal unipotent subgroup of $G$). Let $W$ be the Weyl group of $G$ with respect to $T$, let $R$ be the root system of $(G,T)$, $S$ be the simple roots set with respect to $B$. The group $W$ contains root reflections $r_{\alpha}$ ($\alpha\in R$) acting on $\mathfrak{X}(T)$ as $r_{\alpha}(\lambda)=\lambda-\langle \lambda, \alpha^{\vee}\rangle\alpha$ and is generated by reflections corresponding to simple roots.

Note that parabolic subgroups containing a given Borel subgroup $B$ are para\-met\-ri\-zed by subsets of simple roots $I\subset S$. The Lie algebra of the respective
parabolic subgroup $P=P_{I}$ is $$Lie(P)=\mathfrak{t}\oplus\bigoplus\limits_{\alpha\in R_{+}\cup R_I}\mathfrak{g}_{\alpha}$$ where $R_I\subset R$ is the root subsystem spanned by $I$. 

There is a unique Levi decomposition $P=P_{u}\leftthreetimes L$ where $P_{u}$ is the unipotent radical of $P$ and $L=L_{I}\subset P_{I}$ is a unique Levi subgroup containing a given maximal torus $T\subset B$. The root system of $(L_{I},T)$ is $R_{I}$.

The opposite parabolic subgroup $P^{-}=P^{-}_{I}\supset B^{-}$ associated with $I$ intersects $P_{I}$ in $L_I$ and has the Levi decomposition $P^{-}=P^{-}_{u}\leftthreetimes L$ where $P_{u}^{-}$ is the unipotent radical of $P^{-}$.

Now we recall the definition of horospherical variety. 

\begin{definition}\indent
\begin{itemize}
\item A homogeneous $G$-variety $\Omega$ is called horospherical if the stabilizer of any point in $\Omega$ contains a maximal unipotent subgroup of $G$. 
\item A normal almost homogeneous complex algebraic $G$-variety with the open orbit $\Omega$ is called horospherical if $\Omega$ is horospherical. 
\end{itemize}
\end{definition}

Note that a horospherical variety is spherical in the sence of Definition \ref{defspher}. Let $H$ be the stabilizer of any point $o$ in $\Omega$, thus $G/H\cong\Omega, g\mapsto g.o$. We may assume $H\supset U^{-}$. By \cite[Lemma 7.4]{Tim} we have $H=P^{-}_{u}\leftthreetimes L_{0}$ for a certain parabolic subgroup $P\supset B$ (namely, $P$ is such that $P^{-}=N_{G}(H)$) with the Levi subgroup $L\supset L_{0}\supset L'$ (here $L'$ is the commutator subgroup of $L$) and the unipotent radical $P_u$ of $P$. We may assume that $L\supset T$. 

Note that the injective map $\iota\colon M_\RR\hookrightarrow \XX(T)\otimes\RR$ induces the surjective map $$\iota^{*}\colon \XX^{*}(T)\otimes\RR\twoheadrightarrow N_{\RR}:=N\otimes\RR.$$ 

Now, we recall some facts about horospherical homogeneous $G$-varieties (see \cite[Section 28.1]{Tim}).
\begin{remark}\label{rem2}\indent
\begin{itemize}
\item $M\cong\XX(A)$, where $A\cong P^{-}/H\cong L/L_0\cong T/(T\cap L_0)$;
\item $\mathcal{V}(\Omega)=N_{\mathbb{Q}}$;
\item $B$-stable divisors of $\Omega$ are of the form $D_{\alpha}=\overline{Br_{\alpha}o}$, $\alpha\in S\setminus I$, where $I\subset S$ is the simple root set of $L$. Moreover, $\iota^{*}(\alpha^{\vee})=a_{D_{\alpha}}$.
\end{itemize}
\end{remark}

Thus we obtain the following convex geometric criterion for horospherical varieties. 

\begin{corollary}\label{corhorosph2}
Let $X_{\Sigma}$ be a noncompact horospherical $G$-variety with the open $G$-orbit $\Omega$ and with the colored fan $\Sigma$ such that $N_{\mathbb{R}}\setminus |\Sigma|$ is a connected set. Then $X_{\Sigma}$ admits the Hartogs phenomenon if and only if $$\mathbb{R}_{\geq 0}\langle\overline{N_{\mathbb{R}}\setminus |\Sigma|}\cup \iota^{*}((S\setminus I)^{\vee})\rangle=N_\mathbb{R}$$ where $(S\setminus I)^{\vee}=\{\alpha^{\vee}\mid\alpha\in S\setminus I\}$.
\end{corollary} 

\begin{remark}\label{remhorospher}
Using Corollary \ref{corhorosph2} or Remark \ref{remspher} we obtain that every horospherical homogeneous $G$-variety $\Omega$ with $rk(N)>1$ admits the Hartogs phenomenon. For a horospherical homogeneous $G$-variety $\Omega$ of rank one we have two cases. If $\iota^{*}((S\setminus I)^{\vee})\neq\{0\}$ then there always exists a spherical embedding $X\supset\Omega$ such that $N_{\mathbb{R}}\setminus |\Sigma_{X}|$ is connected and $\mathbb{R}_{\geq 0}\langle\overline{N_{\mathbb{R}}\setminus |\Sigma_{X}|}\cup \iota^{*}((S\setminus I)^{\vee})\rangle=N_\mathbb{R}$. In this case, since $X$ admits the Hartogs phenomenon, then so does $\Omega$. If $\iota^{*}((S\setminus I)^{\vee})=\{0\}$, then $\Omega=\mathbb{C}^{*}\times\Omega_0$, where $\Omega_{0}$ is a compact spherical homogeneous $G$-variety (namely, $\Omega_0\cong G/P^{-}$). In this case $\Omega$ does not admit the Hartogs phenomenon.
\end{remark}
\begin{remark}
From Remarks \ref{remspher} and \ref{remhorospher} we obtain that the Hartogs phenomenon holds for every noncompact spherical homogeneous $G$-variety, except for $\mathbb{C}^{*}\times G/P^{-}$.
\end{remark}

Futher, for $\Omega=G/U^{-}$ we have $I=\emptyset$ and $N=\mathfrak{X}^{*}(T)$; thus we obtain the following

\begin{corollary}\label{cor3}
Let $X_{\Sigma}$ be a horospherical $G$-variety with the open $G$-orbit $\Omega=G/U^{-}$ and with the colored fan $\Sigma$ such that $(\mathfrak{X}^{*}(T)\otimes\mathbb{R})\setminus |\Sigma|$ is a connected set. Then $X_{\Sigma}$ admits the Hartogs phenomenon if and only if $$\mathbb{R}_{\geq 0}\langle\overline{(\mathfrak{X}^{*}(T)\otimes\mathbb{R})\setminus |\Sigma|}\cup \iota^{*}(S^{\vee})\rangle=\mathfrak{X}^{*}(T)\otimes\mathbb{R}.$$
\end{corollary} 

\section{Examples}

In this section we consider examples of horospherical varieties with the open orbit $(SL(2)\times\mathbb{C}^{*})/U^{-}$ where $U^{-}$ is the maximal unipotent subgroup of $SL(2)\times\mathbb{C}^{*}$ consisting of lower triangular matrices with ones on the diagonal. 

\subsection{Some calculations for $(SL(2)\times\mathbb{C}^{*})/U^{-}$}\label{calc1}\indent

Consider $G=SL(2)\times\mathbb{C}^{*}=\big\{ \left(\begin{smallmatrix}
a_{11} & a_{12} & 0\\
a_{21} & a_{22} & 0\\
0& 0 & a_{33}
\end{smallmatrix} \right)\mid a_{11}a_{22}-a_{12}a_{21}=1 ,a_{33}\in\mathbb{C}^{*}\big\}$, and let $$B=\big\{\left(\begin{smallmatrix}
a_{11} & a_{12} & 0\\
0 & a_{22} & 0\\
0& 0 & a_{33}
\end{smallmatrix} \right)\mid a_{11}a_{22}=1 ,a_{33}\in\mathbb{C}^{*}\big\}.$$

Consider the following subgroups of $G$:
\begin{itemize}
\item the maximal torus: $$T=\big\{\left(\begin{smallmatrix}
a_{11} & 0 & 0\\
0 & a_{22} & 0\\
0& 0 & a_{33}
\end{smallmatrix} \right)\mid a_{11}a_{22}=1 ,a_{33}\in\mathbb{C}^{*}\big\}$$
\item the opposite Borel subgroup: $$B^{-}=\big\{\left(\begin{smallmatrix}
a_{11} & 0 & 0\\
a_{21} & a_{22} & 0\\
0& 0 & a_{33}
\end{smallmatrix} \right)\mid a_{11}a_{22}=1 ,a_{33}\in\mathbb{C}^{*}\big\}$$
\item the maximal unipotent subgroup: $$U=\big\{\left(\begin{smallmatrix}
1 & a_{12} & 0\\
0 & 1 & 0\\
0& 0 & 1
\end{smallmatrix} \right)\big\}$$
\item the opposite maximal unipotent subgroup: 
$$U^{-}=\big\{\left(\begin{smallmatrix}
1 & 0 & 0\\
a_{21} & 1 & 0\\
0& 0 & 1
\end{smallmatrix} \right)\big\}$$
\end{itemize}

Note that $\XX(T)=\mathbb{Z}^{2}$. Denoting $t=\left(\begin{smallmatrix}
t_{1} & 0 & 0\\
0 & t_{2} & 0\\
0& 0 & t_{3}
\end{smallmatrix} \right)$, we obtain that each character $\lambda\in\XX(T)$ is given by $\lambda(t)=t_{1}^{l}t_{3}^{m}$ for some $(l,m)\in \mathbb{Z}^{2}$. 

For the Lie algebra $\mathfrak{g}=Lie(G)$ we have the following root decomposition: $$\mathfrak{g}=\mathfrak{t}\oplus\mathfrak{g}_{\alpha_{12}}\oplus\mathfrak{g}_{\alpha_{21}}$$ where 

\begin{itemize}
\item $\mathfrak{g}_{\alpha_{12}}=\mathbb{C}E_{12}$, where $E_{12}=\left(\begin{smallmatrix}
0 & 1 & 0\\
0 & 0 & 0\\
0 & 0 & 0
\end{smallmatrix} \right)$, and $\alpha_{12}\in\XX(T)$ given by $$\alpha_{12}(t)=t_{1}^{2}, \text{ i.e. } \alpha_{12}=(2,0).;$$
\item$\mathfrak{g}_{\alpha_{21}}=\mathbb{C}E_{21}$, where $E_{21}=\left(\begin{smallmatrix}
0 & 0 & 0\\
1 & 0 & 0\\
0 & 0 & 0
\end{smallmatrix} \right)$, and $\alpha_{21}\in\XX(T)$ given by $$\alpha_{12}(t)=t_{1}^{-2}, \text{ i.e. } \alpha_{21}=(-2,0).$$
\end{itemize}

The root system of $G$ is $R=\{\alpha_{12},\alpha_{21}\}$. Since $Lie(B)=\mathfrak{t}\oplus \mathbb{C}E_{12}$, it follows that the subset of simple roots with respect to $B$ is $S=\{\alpha_{12}\}$. 

The Weyl group of $G$ with respect to $T$ is $W=N_{G}(T)/T=\{e, r_{12}\}$ where $r_{12}$ is the reflection which is represented by the matrix $U=E_{12}-E_{21}+E_{33}\in N_{G}(T)$. 

Now, we consider the homogeneous $G$-variety $G/U^{-}$. We have the isomorphism $\phi\colon G/U^{-}\cong \mathbb{C}^{2}\setminus\{(0,0)\}\times\mathbb{C}^{*}$ induced by the map $$G\to \mathbb{C}^{2}\setminus\{(0,0)\}\times\mathbb{C}^{*}, \left(\begin{smallmatrix}
a_{11} & a_{12} & 0\\
 a_{21} & a_{22} & 0\\
0& 0 & a_{33}
\end{smallmatrix} \right)\mapsto (a_{12},a_{22},a_{33}).$$

The isomorphism $\phi$ is $G$-equivariant with respect to the left multiplication action $G$ on $G/U^{-}$ and the standard action $G$ on $\mathbb{C}^{2}\setminus\{(0,0)\}\times\mathbb{C}^{*}$. 

Using the isomorphism $\phi$, we obtain that $B$-orbits of $G/U^{-}$ are of the form:

\begin{itemize}
\item unique $B$-stable divisor $$D_{12}=\{(x_{1},x_{2},w)\in \mathbb{C}^{2}\setminus\{(0,0)\}\times\mathbb{C}^{*})\mid x_{2}=0\};$$ 
\item open $B$-orbit $$O=\{(x_{1},x_{2},w)\in \mathbb{C}^{2}\setminus\{(0,0)\}\times\mathbb{C}^{*}\mid x_{2}\neq0\}.$$
\end{itemize}

Now we calculate the weight lattice. The algebra of regular functions on $\Omega$ is the following $$\mathbb{C}[\Omega]=\mathbb{C}[\mathbb{C}^{2}\setminus\{(0,0)\}\times\mathbb{C}^{*}]=\bigoplus\limits_{(i,j,k)\in(\mathbb{Z}_{\geq0})^{2}\times\mathbb{Z}}\mathbb{C}x_{1}^{i}x_{2}^{j}w^{k}.$$ It follows that $x_{2}^{l}w^{-m}\in\mathbb{C}(\Omega)^{(B)}_{(l,m)}$ where $\lambda=(l,m)\in \XX(T)=\mathbb{Z}^{2}$.

Thus we see that $M=\XX(T)$ (of course, we may apply Remark \ref{rem2}; in this case, $P=B,H=U^{-}$ and $A\cong T$). 

The $B$-stable divisor $D_{12}$ defines the discrete valuation $$v_{12}\colon \mathbb{C}(\Omega)\setminus\{0\}\to\mathbb{Z}$$ and the point $a_{D_{12}}\in N$ by the formula $\langle a_{D_{12}},(l,m)\rangle:=v_{D_{12}}(x_{2}^{l}w^{-m})=l$. This means that $a_{D_{12}}=(1,0)$.

Let $(,)$ be a standard scalar product in $\XX(T)\otimes\mathbb{R}$. The dual simple root $\alpha_{12}^{\vee}\in \XX^{*}(T)\otimes\mathbb{R}$ given by formula $$\langle\alpha_{12}^{\vee},\lambda\rangle=\frac{2(\alpha_{12},\lambda)}{(\alpha_{12},\alpha_{12})}$$ for all $\lambda\in\XX(T)$ It follows that $\alpha_{12}^{\vee}=(1,0)$. Thus $\alpha_{12}^{\vee}=a_{D_{12}}$ (which actually follows from Remark \ref{rem2}).

\subsection{$(SL(2)\times \CC^{*})/U^{-}$-embeddings}\indent

Using Corollary \ref{cor3} we obtain the following criterion

\begin{corollary}\label{crit}
Let $X_{\Sigma}$ be a horospherical variety with the open orbit $(SL(2)\times\CC^{*})/U^{-}$ and with a colored fan $\Sigma$ such that $\mathbb{R}^{2}\setminus |\Sigma|$ is connected. Then $X_{\Sigma}$ admits the Hartogs phenomenon if and only if $$\mathbb{R}_{\geq 0}\langle \overline{\mathbb{R}^{2}\setminus |\Sigma|}\cup \{\alpha_{12}^{\vee}\}\rangle=\mathbb{R}^{2}.$$
\end{corollary}

Now we construct examples of horospherical varieties and the corresponding colored fans. We apply Corollary \ref{crit}.

\begin{remark}
The standard action of $G=SL(2)$ on $\mathbb{C}^{2}$ induces the action of $SL(2)$ on the projective space $\mathbb{P}^{2}=\mathbb{P}(\mathbb{C}\oplus\mathbb{C}^{2})$ which is given by the formula $$\left(\begin{smallmatrix}
a_{11} & a_{12}\\
a_{21} & a_{22}
\end{smallmatrix} \right).[z_{0}:z_{1}:z_{2}]:=[z_{0}:(a_{11}z_{1}+a_{12}z_{2}):(a_{21}z_{1}+a_{22}z_{2})].$$ 

By $B$ we denote the subgroup of $G$ consisting of upper triangular matrices. 

Note that $G$-orbits of this action are the following:
\begin{itemize}
\item $G$-orbit $H_{\infty}=\{[z_{0}:z_{1}:z_2]\in\mathbb{P}^{2}\mid z_{0}=0\}$. It is a unique $G$-stable divisor;
\item $G$-orbit $H_{0}=\{[1:0:0]\}$;
\item the open $G$-orbit $\mathbb{P}^{2}\setminus(H_{\infty}\cup H_{0})\cong\mathbb{C}^{2}\setminus\{(0,0)\}$;
\end{itemize}

Also we have a $B$-stable but not a $G$-stable divisor $$H_{12}=\{[z_{0}:z_{1}:z_2]\in\mathbb{P}^{2}\mid z_{2}=0\}.$$ 

Note that the open $B$-orbit is $\mathbb{P}^{2}\setminus(H_{\infty}\cup H_{0}\cup H_{12})\cong\mathbb{C}^{2}\setminus\{x_{2}=0\}$ (where $x_{2}=\frac{z_{2}}{z_{0}}$).

\end{remark}

\begin{remark}
Consider the toric action of $\mathbb{C}^{*}$ on $\mathbb{P}^{1}$ which is given by the formula $$t.[w_{0}:w_{1}]=[w_{0}:tw_{1}].$$

In this case, we have the following $\mathbb{C}^{*}$-orbits:
\begin{itemize}
\item The open $\mathbb{C}^{*}$-orbit is $\mathbb{P}^{1}\setminus\{[1:0],[0:1]\}\cong\mathbb{C}^{*}=\mathbb{C}\setminus\{w=0\}$ (where $w=\frac{w_1}{w_0}$);
\item $\mathbb{C}^{*}$-divisors are $[1:0], [0:1]$.
\end{itemize}

\end{remark}

Now we consider the compact variety $X'=\mathbb{P}^{2}\times\mathbb{P}^{1}$ with the action of $SL(2)\times\mathbb{C}^{*}$ given by the formula
\begin{multline*}
\left(\begin{smallmatrix}
a_{11} & a_{12} & 0\\
a_{21} & a_{22} & 0\\
0 & 0 & a_{33}
\end{smallmatrix} \right).([z_{0}:z_{1}:z_{2}],[w_{0}:w_{1}])=\\=([z_{0}:(a_{11}z_{1}+a_{12}z_{2}):(a_{21}z_{1}+a_{22}z_{2})],[w_{0}:a_{33}w_{1}]).
\end{multline*}

In this case, the open $SL(2)\times\mathbb{C}^{*}$-orbit is $\mathbb{C}^{2}\setminus\{(0,0)\}\times\mathbb{C}^{*}\cong SL(2)\times\mathbb{C}^{*}/U^{-}$ and the open $B$-orbit is $\mathbb{C}^{2}\setminus\{(x_{2}=0)\}\times\mathbb{C}^{*}$.

We have the followng $SL(2)\times\mathbb{C}^{*}$-orbits which are not open:

\begin{itemize}
\item 0-dimensional: $H_{0}\times[1:0]$, $H_{0}\times[0:1]$;
\item 1-dimensional: $H_{\infty}\times [1:0], H_{\infty}\times [0:1]$, $H_{0}\times\mathbb{C}^{*}$;
\item 2-dimensional: $H_{\infty}\times\mathbb{C}^{*}, \mathbb{C}^{2}\setminus\{(0,0)\}\times[1:0]$, $\mathbb{C}^{2}\setminus\{(0,0)\}\times[0:1]$;
\end{itemize}

Also we have the following $B$-stable divisors: 

\begin{itemize}
\item $SL(2)\times\mathbb{C}^{*}$-stable divisors: $D_{\infty}:=H_{\infty}\times\mathbb{P}^{1}, D_{10}:=\mathbb{P}^{2}\times[1:0], D_{01}:=\mathbb{P}^{2}\times[0:1]$;
\item not $SL(2)\times\mathbb{C}^{*}$-stable divisor: $D_{12}:=H_{12}\times\mathbb{P}^{1}$.
\end{itemize}

Now we calculate points $a_{D_\infty},a_{D_{10}},a_{D_{01}}, a_{D_{12}}$ in $N=\mathbb{Z}^{2}$ corresponding to the $B$-stable divisors $D_{\infty}, D_{10}, D_{01}, D_{12}$. 

Consider a character $\lambda=(l,m)\in M=\mathbb{Z}^{2}$. We have the following: 
\begin{itemize}
\item $\langle a_{D_\infty},(l,m)\rangle=v_{D_{\infty}}(x_{2}^{l}w^{-m})=v_{D_{\infty}}((\frac{z_2}{z_0})^{l}w^{-m})=-l$; thus $a_{D_\infty}=(-1,0)$;
\item $\langle a_{D_{10}},(l,m)\rangle=v_{D_{10}}(x_{2}^{l}w^{-m})=v_{D_{10}}(x_2^{l}(\frac{w_1}{w_0})^{-m})=m$; thus $a_{D_{10}}=(0,1)$
\item $\langle a_{D_{01}},(l,m)\rangle=v_{D_{01}}(x_2^{l}(\frac{w_1}{w_0})^{-m})=-m$; thus $a_{D_{10}}=(0,-1)$
\item $\langle a_{D_{12}},(l,m)\rangle=v_{D_{12}}(x_{2}^{l}w^{-m})=v_{D_{12}}((\frac{z_2}{z_0})^{l}w^{-m})=l$; thus $a_{D_{12}}=(1,0)=\alpha_{12}^{\vee}$;
\end{itemize}

So, we obtain the colored fan $\Sigma'$ of $X'$ consisting of the following colored cones (see Figure \ref{3fan}): 
\begin{itemize}
\item 0-dimensional cone: $((0,0),\emptyset)$;
\item 1-dimensional cones: $(\mathbb{R}_{\geq0}\langle (-1,0)\rangle, \emptyset)$, $(\mathbb{R}_{\geq0}\langle (0,-1)\rangle, \emptyset)$, $(\mathbb{R}_{\geq0}\langle (0,1)\rangle, \emptyset)$, $(\mathbb{R}_{\geq0}\langle (1,0)\rangle, D_{12})$; 
\item 2-dimensional cones: $(\mathbb{R}_{\geq0}\langle (-1,0),(0,1)\rangle, \emptyset)$, $(\mathbb{R}_{\geq0}\langle (-1,0),(0,-1)\rangle, \emptyset)$, 

$(\mathbb{R}_{\geq0}\langle (1,0), (0,1)\rangle, D_{12})$, 
$(\mathbb{R}_{\geq0}\langle (1,0), (0,-1)\rangle, D_{12})$; 
\end{itemize}
\begin{figure}[H]
\begin{center}
\begin{minipage}[H]{0.3\linewidth}
\begin{center}
		\begin{tikzpicture}[scale=0.7]
		\draw[help lines] (-3,-3) grid (3,3);
		\draw (-3,0) coordinate (A) - - (0,0) coordinate (B) -- (0,-3) coordinate (C) pic [pattern=horizontal lines ,angle radius=2cm] {angle};
		\draw (0,-3) coordinate (A) - - (0,0) coordinate (B) -- (3,0) coordinate (C) pic [pattern=vertical lines ,angle radius=2cm] {angle};
		\draw (3,0) coordinate (A) - - (0,0) coordinate (B) -- (0,3) coordinate (C) pic [pattern=dots ,angle radius=2cm] {angle};
		\draw (0,3) coordinate (A) - - (0,0) coordinate (B) -- (-3,0) coordinate (C) pic [pattern=north west lines ,angle radius=2cm] {angle};
		\fill[fill=blue](1,0) circle (4pt) node[above right, blue] {$\alpha_{12}^{\vee}$};
		\end{tikzpicture}
\end{center}
\end{minipage}
\caption{Colored fan $\Sigma'$ of $\mathbb{P}^{2}\times\mathbb{P}^{1}$. \label{3fan}}
\end{center}
\end{figure}

Now we delete from $X'$ some $SL(2)\times\mathbb{C}^{*}$-stable divisors to obtain a noncompact horospherical variety $X$ and we apply for $X$ the convex geometric criterion in Corollary \ref{crit}. 

\begin{itemize}
\item For $X:= X'\setminus D_{\infty}\cong \mathbb{C}^{2}\times\mathbb{P}^{1}$ we have the colored fan $\Sigma$ as in Figure \ref{3fan1}. Since $\mathbb{R}_{\geq0}\langle\overline{\mathbb{R}^{2}\setminus |\Sigma|}\cup\{\alpha_{12}^{\vee}\}\rangle=\mathbb{R}^{2}$, it follows that $X$ admits the Hartogs phenomenon.  

\begin{figure}[H]
\begin{center}
\begin{minipage}[H]{0.3\linewidth}
\begin{center}
		\begin{tikzpicture}[scale=0.7]
		\draw[help lines] (-3,-3) grid (3,3);
		\draw (0,-3) coordinate (A) - - (0,0) coordinate (B) -- (3,0) coordinate (C) pic [pattern=vertical lines ,angle radius=2cm] {angle};
		\draw (3,0) coordinate (A) - - (0,0) coordinate (B) -- (0,3) coordinate (C) pic [pattern=dots ,angle radius=2cm] {angle};
		\fill[fill=blue](1,0) circle (4pt) node[above right, blue] {$\alpha_{12}^{\vee}$};
		\end{tikzpicture}
\end{center}
\end{minipage}
\caption{Colored fan $\Sigma$ of $\mathbb{C}^{2}\times\mathbb{P}^{1}$. \label{3fan1}}
\end{center}
\end{figure}

\item For $X:= X'\setminus D_{10}\cong \mathbb{P}^{2}\times\mathbb{C}^{1}$ we have the colored fan $\Sigma$ as in Figure \ref{3fan2}. Since $\mathbb{R}_{\geq0}\langle\overline{\mathbb{R}^{2}\setminus |\Sigma|}\cup\{\alpha_{12}^{\vee}\}\rangle\neq\mathbb{R}^{2}$, it follows that $X$ does not admit the Hartogs phenomenon.  

\begin{figure}[H]
\begin{center}
\begin{minipage}[H]{0.3\linewidth}
\begin{center}
		\begin{tikzpicture}[scale=0.7]
		\draw[help lines] (-3,-3) grid (3,3);
		\draw (-3,0) coordinate (A) - - (0,0) coordinate (B) -- (0,-3) coordinate (C) pic [pattern=horizontal lines ,angle radius=2cm] {angle};
		\draw (0,-3) coordinate (A) - - (0,0) coordinate (B) -- (3,0) coordinate (C) pic [pattern=vertical lines ,angle radius=2cm] {angle};
		\fill[fill=blue](1,0) circle (4pt) node[above right, blue] {$\alpha_{12}^{\vee}$};
		\end{tikzpicture}
\end{center}
\end{minipage}
\caption{Colored fan $\Sigma$ of $\mathbb{P}^{2}\times\mathbb{C}$. \label{3fan2}}
\end{center}
\end{figure}

\end{itemize}

\end{document}